\documentclass[a4paper, 11pt, twoside, notitlepage]{amsart}

\usepackage{amsmath,amscd}
\usepackage{amssymb}
\usepackage{amsthm}
\usepackage{comment}
\usepackage{graphicx, xcolor}

\usepackage{mathrsfs}
\usepackage[ocgcolorlinks, linkcolor=blue]{hyperref}

\usepackage{bm}
\usepackage{bbm}
\usepackage{url}

\usepackage[utf8]{inputenc}
\usepackage{mathtools,amssymb}
\usepackage{esint}
\usepackage{tikz}
\usepackage{dsfont}
\usepackage{relsize}
\usepackage{url}
\usepackage{xcolor}
\usepackage{graphicx}
\usepackage{mathrsfs}
\usepackage[shortlabels]{enumitem}
\usepackage{lineno}
\usepackage{amsmath}
\usepackage{enumitem}
\usepackage{amsthm} 
\usepackage{verbatim}
\usepackage{dsfont}
\numberwithin{equation}{section}

\allowdisplaybreaks

\mathtoolsset{showonlyrefs}

\graphicspath{{images/}}

\newtheorem{theorem}{Theorem}[section]
\newtheorem{lemma}[theorem]{Lemma}

\title[Fractional anisotropic Calder\'on problem]{A local uniqueness theorem for the fractional Schr\"odinger equation on closed Riemannian manifolds}

\author[Y.-H.~Lin]{Yi-Hsuan Lin}
\address{Department of Applied Mathematics, National Yang Ming Chiao Tung University, Hsinchu, Taiwan \& Fakult\"at f\"ur Mathematik, University of Duisburg-Essen, Essen, Germany}
\email{yihsuanlin3@gmail.com}

\newcommand{\N}{{\mathbb N}}

\newcommand {\dVg} {\mathsf{dV}_g}

\newcommand{\LC}{\left(}
\newcommand{\RC}{\right)}
\newcommand{\wt}{\widetilde}

\begin{document}

	\maketitle
	\begin{abstract}
	We investigate that a potential $V$ in the fractional Schr\"odinger equation $\LC \LC -\Delta_g \RC^s +V \RC u=f$ can be recovered locally by using the local source-to-solution map on smooth connected closed Riemannian manifolds. To achieve this goal, we derive a related new Runge approximation property.
	
		\medskip
		
		

	\end{abstract}
	
	
	\section{Introduction}\label{sec: introduction}

	Let $(M,g)$ be a smooth closed Riemannian manifold of dimension $\dim M\geq 2$, and $-\Delta_g$ be the Laplace-Beltrami operator defined on $M$. 
	It is known that the domain of $-\Delta_g$ is the standard Sobolev space $H^2(M)$. Let us also denote the eigenvalue of $-\Delta_g$ on $M$ with increasing order, by $0=\lambda_0<\lambda_1\leq \lambda_2\leq \ldots$, and the first eigenfunction $\phi_0\equiv 1$ (up to normalization). Consider $\LC\phi_{k}\RC_{k\in \N}$ to be an orthonormal basis for $\ker\LC -\Delta_g -\lambda_k\RC$ in $L^2(M)$ (with multiplicity), with respect to $\lambda_k$, for $k\in \N$. 
	Let us define the $L^2$ inner product on $M$ as $\LC f,g\RC_{L^2(M)}:=\int_M fg\, \dVg$, where $\dVg$ denotes the natural volume element on $(M,g)$.
	Given $s\in (0,1)$, the fractional Laplace-Beltrami on $M$ of order $s$ as an unbounded self-adjoint operator can be defined via the spectral theorem
	\begin{equation}
		\begin{split}
			\LC -\Delta_g \RC^s u(x) = \sum_{k=0}^\infty \lambda_k^{s} \LC u, \phi_k\RC_{L^2(M)}\phi_k(x)
		\end{split}
	\end{equation}
	in the domain 
	\begin{equation}
		\begin{split}
			\mathcal{D}\LC \LC -\Delta_g \RC^s \RC : = \bigg\{u\in L^2(M) : \, \sum_{k=0}^\infty \lambda^{2s}_k \big|  \LC u, \phi_k\RC_{L^2(M)} \big|^2<\infty \bigg\}=H^{2s}(M).
		\end{split}
	\end{equation}
	The fractional Laplace-Beltrami operator $\LC -\Delta_g\RC^s$ can be viewed as a linear map from $H^{a}(M)\to H^{a-2s}(M)$, for any $a\geq 0$.
	
	Let us characterize our mathematical model in this work. Let $(M,g)$ be a smooth closed Riemannian manifold, and $V=V(x)\in C^\infty(M)$ be a nonnegative potential. Consider 
	\begin{equation}\label{equ: main}
		\LC \LC -\Delta_g \RC^s+V \RC u =f \text{ in }M. 
	\end{equation}
	Via the eigenvalue condition, it is also known that $\LC (-\Delta_g) u , 1 \RC_{L^2(M)}=0$.
	With the above assumptions, by standard regularity theory, it is not hard to see that there exists a unique solution $u_f \in C^\infty(M)$ of the equation \eqref{equ: main}. Let $\mathcal{O}\subset M$ be a nonempty open set and $f\in C^\infty_c(\mathcal{O})$. Then one can define the local \emph{source-to-solution map} $L_{V,\mathcal{O}} $ of \eqref{equ: main} via 
	\begin{equation}\label{local source-to-solution map}
		L_{V,\mathcal{O}} : C^\infty_c(\mathcal{O}) \ni f \mapsto \left. u_f \right|_{\mathcal{O}}\in C^\infty(\mathcal{O}),
	\end{equation}
	where $u_f\in C^\infty(M)$ is the solution to \eqref{equ: main}. We want to recover the potential $V$ by using the local measurement $L_{V,\mathcal{O}}$. We can prove the next theorem:
	
	\begin{theorem}\label{Thm: potential}
		Given $s\in (0,1)$, let $(M,g)$ be smooth closed connected Riemannian manifolds of dimension $n \geq 2$, and $0\leq V_j\in C^\infty(M)$, for $j=1,2$. Let $\mathcal{O}\Subset M$ be an open set. Suppose that either $V_1=V_2$ in $\mathcal{O}$ or $V_1=V_2$ in $M\setminus \mathcal{O}$, then 
		\begin{equation}\label{s-t-s potential}
			L_{V_1,\mathcal{O}}(f)=	L_{V_2,\mathcal{O}} (f), \text{ for all }f \in C^\infty_c(\mathcal{O}).
		\end{equation}
		implies $V_1=V_2$ in $M$.
	\end{theorem}
	
	The proof of Theorem \ref{Thm: potential} relies on suitable integral identity and approximation property. Meanwhile, a variant result of Theorem \ref{Thm: potential} has been addressed in \cite{GLX} by using exterior measurements
	Throughout this paper, we always denote $(M,g)$ as a smooth, connected closed Riemannian manifold of $\dim M\geq 2$. 
	
	\medskip
	
	\noindent $\bullet$ \textbf{Related works.} Inverse problems for fractional/nonlocal equations have been paid attention in recent years. In the pioneering work \cite{GSU20}, the authors first proved a global uniqueness result of bounded potentials for the fractional Schr\"odinger equation by using the exterior Dirichlet-to-Neumann (DN) map. Afterward, there are many related results appeared in this direction, and we refer readers to  \cite{CLL2017simultaneously,GLX,CLL2017simultaneously,cekic2020calderon,CRTZ-2022,feizmohammadi2021fractional,FKU2024calder,KLW2021calder,RS17,RZ2022unboundedFracCald,LZ2024NDOT,GRSU18,KLZ-2022,KRZ-2023}). Let us point out that the essential methods in this area are \emph{unique continuation property} and \emph{Runge approximation}. More recently, one could also solve (local) inverse problems by using the nonlocal method (see \cite{LNZ2024}).

	\section{Preliminaries}\label{sec: prel}
	
	We first review a simple fact.
	
	\begin{lemma}\label{Lemma: integration by parts}
		There holds 
		\begin{equation}\label{equ: integration by parts}
			\begin{split}
				\LC  \LC -\Delta_g \RC^s u ,v  \RC_{H^{-s}(M)\times H^s(M)} &= \big( \LC -\Delta_g\RC^{s/2}u, \LC -\Delta_g \RC^{s/2}v \big)_{L^2(M)} \\
				&=	\LC u,  \LC -\Delta_g \RC^s v  \RC_{H^s(M)\times H^{-s}(M)},
			\end{split}
		\end{equation}
		for all $u,v \in H^s(M)$, where $H^{-s}(M)$ denotes the dual space of $H^s(M)$.
	\end{lemma}
	
	\begin{proof}
		This is followed directly by straightforward computations with eigenvalue and eigenfunction decomposition.
	\end{proof}
	
	\begin{lemma}[Well-posedness]\label{Lemma: well-posed}
		Given $f\in H^{-s}(M)$, let $0\leq V\in C^\infty(M)$, then there exists a unique solution $u\in H^s(M)$ to the equation \eqref{equ: main}. Moreover, if $f\in C^\infty(M)$, then $u\in C^\infty(M)$.
	\end{lemma}
	
	\begin{proof}
		Consider the bilinear form associated with \eqref{equ: main} to be 
		\begin{equation}
			B_{g,V}(u,v):= \big( \LC -\Delta_g\RC^{s/2} u , \LC -\Delta_g \RC^{s/2} v \big)_{L^2(M)}+\LC Vu,v\RC_{L^2(M)},
		\end{equation}
		for any $u,v\in H^s(M)$.
		Using $V\geq 0$ in $M$, the standard Lax-Milgram theorem and later statement can be seen by using elliptic regularity theory.
	\end{proof}

	With the above lemma at hand, one can define the local source-to-solution map $L_{V,\mathcal{O}}$ of \eqref{equ: main} rigorously, which is exactly given by \eqref{local source-to-solution map}. We next derive a useful integral identity.

	\begin{lemma}
		Let $V,V_j\in C^\infty(M)$, for $j=1,2$, and $\mathcal{O}\subset M$ be a nonempty open set, then 
		\begin{itemize}
			\item[(i)] $\ $ (Self-adjointness). There holds 
			\begin{equation}\label{self-adjoint}
				\LC f_1 , L_{V,\mathcal{O}}f_2 \RC_{L^2(M)}=	\LC  L_{V,\mathcal{O}}f_1, f_2 \RC_{L^2(M)},
			\end{equation}
			
			\item[(ii)]$\ $ (Integral identity). Let $V_1,V_2\in C^\infty(M)$, then there holds that 
			\begin{equation}\label{integral identity}
				\begin{split}
					\LC	\LC  L_{V_1,\mathcal{O}} -  L_{V_2,\mathcal{O}} \RC f_1, f_2 \RC_{L^2(M)}= -\big( \LC V_1 -V_2 \RC u_{1}^{f_1}, u_{2}^{f_2}\big)_{L^2(M)},
				\end{split}
			\end{equation}
			for all $f_1,f_2 \in C^\infty_c (\mathcal{O})$, where $u_j^{f_j}$ are the solutions to 
			\begin{equation}\label{equ: main j=1,2}
				\LC \LC -\Delta_g\RC^s +V_j\RC u_j^{f_j}=f_j \text{ in }M, \text{ for }j=1,2.
			\end{equation}
		\end{itemize}
	\end{lemma}

	\begin{proof}
		We first prove the self-adjointness of the local source-to-solution map $L_{V,\mathcal{O}}$. For any $f_1,f_2\in C^\infty_c(\mathcal{O})$, direct computations imply that 
		\begin{equation}
			\begin{split}
				\LC L_{V,\mathcal{O}}f_1, f_2 \RC_{L^2(M)} &= \LC L_{V,\mathcal{O}}f_1, f_2 \RC_{L^2(\mathcal{O})} = \LC u_1^{f_1},f_2 \RC_{L^2(M)}\\
				&= \LC u_1^{f_1}, \LC \LC -\Delta_g \RC^s +V \RC u_2^{f_2}\RC_{L^2(M)}\\
				&= \underbrace{ \LC  \LC \LC -\Delta_g \RC^s +V \RC u_1^{f_1}, u_2^{f_2}\RC_{L^2(M)}}_{\text{By Lemma \ref{Lemma: integration by parts}}} \\
				&=\LC f_1, L_{V,\mathcal{O}} f_2 \RC_{L^2(M)} .
			\end{split}
		\end{equation}
		By using the adjointness \eqref{self-adjoint}, the integral identity \eqref{integral identity} can be seen via 
		\begin{equation}
			\begin{split}
				&\quad \, \LC	\LC  L_{V_1,\mathcal{O}} -  L_{V_2,\mathcal{O}} \RC f_1, f_2 \RC_{L^2(M)} \\
				&= \LC L_{V_1,\mathcal{O}}  f_1, f_2 \RC_{L^2(M)} - \LC f_1 ,L_{V_2,\mathcal{O}}f_2 \RC_{L^2(M)} \\
				&= \underbrace{\LC u_1, \LC \LC -\Delta_g\RC^s +V_2 \RC u_2\RC_{L^2(M)}- \LC  \LC \LC -\Delta_g\RC^s +V_1 \RC u_1, u_2\RC_{L^2(M)}}_{\text{By \eqref{equ: main j=1,2}}}\\
				&=\underbrace{ -\LC \LC V_1 -V_2 \RC u_{1} , u_{2}\RC_{L^2(M)}}_{\text{By Lemma \ref{Lemma: integration by parts}}},
			\end{split}
		\end{equation}
		which proves the assertion.
	\end{proof}

	Let us also recall a well-known unique continuation property (UCP) for the fractional Laplace-Beltrami operator.
	
	\begin{lemma}[UCP]\label{Lemma: UCP}
		Let $\emptyset \neq \mathcal{O}\subset M$ be an open subset, then $u=\LC -\Delta_g \RC^s u =0 $ in $\mathcal{O}$ implies $u\equiv 0$.
	\end{lemma}
	
	The proof can be found in \cite{GLX,feizmohammadi2021fractional} with different approaches.

	\section{Proof of main results}\label{sec: proof}
	
	To prove Theorem \ref{Thm: potential}, we first show a useful approximation result.
	
	\begin{lemma}[Runge approximation]\label{Lemma: runge}
		Consider the set 
		\begin{equation}
			\mathcal{R}:=\left\{ \left. u_f \right|_{\mathcal{O}} :\, \text{for any }f\in C^\infty_c(\mathcal{O}) \right\},
		\end{equation}
		where $u_f$ is the solution to \eqref{equ: main}. Then $\mathcal{R}$ is dense in $H^s(\mathcal{O})$.
	\end{lemma}
	
	\begin{proof}
	    By Lemma \ref{Lemma: well-posed}, the solution $u_f$ of \eqref{equ: main} belongs to $H^s(M)$ for any $f\in H^{-s}(M)$, then we have $\mathcal{R}\subset H^s(\mathcal{O})$.	    
		Using the Hahn-Banach theorem, consider $\varphi \in \wt H^{-s}(\mathcal{O})= (H^s(\mathcal{O}))^\ast$ (dual space of $ H^s(\mathcal{O})$), which satisfies 
		\begin{equation}
			\varphi \LC u_f \RC=0, \text{ for all } f\in C^\infty_c(\mathcal{O}),
		\end{equation} 
		then we aim to show prove $\varphi \equiv 0$. 	 
		Let $w\in H^s(M)$ be the solution to 
		\begin{equation}\label{adjoint equation}
			\LC \LC -\Delta_g \RC^s +V \RC w =\varphi  \text{ in }M,
		\end{equation}
		then we have 
		\begin{equation}
			\begin{split}
				0&=\left\langle \varphi,  u_f \right\rangle_{\wt H^{-s}(\mathcal{O})\times H^s(\mathcal{O})} =\left\langle \LC \LC -\Delta_g \RC^s +V \RC w ,u_f \right\rangle_{H^{-s}(M)\times H^s(M)}\\
				&= \underbrace{\LC \LC \LC -\Delta_g \RC^s +V \RC u_f, w  \RC_{L^2(M)}}_{\text{By Lemma \ref{Lemma: integration by parts}}} = \LC f,w\RC_{L^2(M)}=\LC f,w\RC_{L^2(\mathcal{O})},
			\end{split}
		\end{equation}
		for any $f\in C^\infty_c(\mathcal{O})$. This implies $w=0$ in $\mathcal{O}$. 
		
		On the other hand, multiplying \eqref{adjoint equation} by the function $w$, one has 
		\begin{equation}
				\begin{split}
					\left\langle \varphi ,w \right\rangle_{H^{-s}(M)\times H^s(M)}&=\int_M w \LC \LC-\Delta_g \RC^s + V \RC w\, \dVg\\
				&=\underbrace{\int_M \big( \big| \LC -\Delta_g \RC^{s/2} w \big|^2 + V |w|^2 \big) \,  \dVg \geq 0}_{\text{By using }V\geq 0 \text{ in }M},
			\end{split}
		\end{equation}
		which implies 
		\begin{equation}
			\begin{split}
				\underbrace{\langle \varphi, w \rangle_{H^{-s}(M)\times H^s(M)}=\langle \varphi, w \rangle_{\wt H^{-s}(\mathcal{O})\times H^s(\mathcal{O})}=0}_{\text{By using }\varphi \in \wt H^{-s}(\mathcal{O}) \text{ and }w=0 \text{ in }\mathcal{O}}.
			\end{split}
		\end{equation} 
		Hence, we have $0\leq \int_M \LC \left| \LC -\Delta_g \RC^s w \right|^2 + V |w|^2 \RC \dVg=0$, such that $w$ must be zero in $M$, which yields that $0=\varphi\in \wt H^{-s}(\mathcal{O})$ by using the adjoint equation \eqref{adjoint equation}.
	\end{proof}

\begin{proof}[Proof of Theorem \ref{Thm: potential}]
	For any $f_1,f_2 \in C^\infty_c (\mathcal{O})$, let $u_j$ be the solution of \eqref{equ: main j=1,2}, for $j=1,2$. Let us first consider $V_1=V_2$ in $M\setminus \mathcal{O}$, combining with \eqref{s-t-s potential} and \eqref{integral identity}, then we can obtain 
	\begin{equation}\label{integral id=0}
		\int_{\mathcal{O}} \LC V_1-V_2 \RC u_{1}^{f_1}u_2^{f_2}  \, \dVg =	\int_{M} \LC V_1-V_2 \RC u_{1}^{f_1}u_2^{f_2}  \, \dVg =0.
	\end{equation}
	By Lemma \ref{Lemma: runge}, given any $h\in H^s(\mathcal{O})$, one can find sequences of functions $\big( f_1^{f_1^{(k)}}\big)_{k\in \N}$, $\big( f_2^{(\ell)}\big)_{\ell \in \N}$ such that the corresponding solutions satisfy $u_1^{{(k)}}\to h$ and $u_2^{f_2^{(\ell)}}\to 1$ in $H^s(\mathcal{O})$, as $k,\ell \to \infty$. Now, inserting these sequences of solutions into \eqref{integral id=0} and taking limits, there holds that 
	\begin{equation*}
		\int_{\mathcal{O}}\LC V_1-V_2 \RC h \, \dVg =0.
	\end{equation*}
	Due to arbitrariness of $h\in H^s(\mathcal{O})$, there holds $V_1=V_2$ in $\mathcal{O}$ so that $V_1=V_2$ in $M$. 
	
	On the other hand, if $V_1=V_2$ in $\mathcal{O}$, then the condition \eqref{s-t-s potential} implies that $(-\Delta_g)^s\big(  u_1^f-u_2^f \big) =u_1^f -u_2^f=0$ in $\mathcal{O}$. The UCP implies $u^f:=u_1^f=u_2^f$ in $M$ for any $f\in C^\infty_c(\mathcal{O})$. This implies $\LC V_1 -V_2 \RC u^f=0 $ in $M\setminus \mathcal{O}$. Similar to the arguments as in \cite{GRSU18,LZ2024NDOT}, by the smoothness of $u^f$ and UCP for $(-\Delta_g)^s$, it is not hard to show $V_1=V_2$ in $M\setminus \mathcal{O}$. This implies $V_1=V_2$ in $M$ as well. 
\end{proof}

\noindent \textbf{Acknowledgment.}
	Y.-H. Lin is partially supported by the National Science and Technology Council (NSTC) Taiwan, under project 113-2628-M-A49-003. Y.-H. Lin is also a Humboldt research fellow.

	
	
	
	\end{document}